\documentclass[11pt, reqno]{amsart}
\usepackage{amsmath, amsthm, amscd, amsfonts, amssymb, graphicx}
\usepackage{ragged2e}
\usepackage[bookmarksnumbered, colorlinks, plainpages]{hyperref}
\usepackage{cite}
\newtheorem{theorem}{Theorem}[section]
\newtheorem{lemma}[theorem]{Lemma}

\newtheorem{Result}[theorem]{Result}
\theoremstyle{definition}
\newtheorem{definition}[theorem]{Definition}

\theoremstyle{remark}

\numberwithin{equation}{section}

\begin{document}
\setcounter{page}{1}

\centerline{}

\centerline{}

\title[]{Extension of Lohwater\,--\,Pommerenke's Theorem for strongly-normal Maps}

\author{Gopal Datt}
\address{Department of Mathematics, Babasaheb Bhimrao Ambedkar University, Lucknow, India}
\email{\textcolor[rgb]{0.00,0.00,0.84}{ggopal.datt@gmail.com, gopal.du@bbau.ac.in}}
\author{Rahul Gogoi}
\address{Department of Mathematics, University of Delhi, Delhi, India}
\email{\textcolor[rgb]{0.00,0.00,0.84}{rgogoi1729@gmail.com, rgogoi@maths.du.ac.in}}


\begin{abstract}
We introduce strong normality for holomorphic curves and logharmonic mappings, extending classical 
normality concepts. We establish an extension of the rescaling characterization due to Lohwater and Pommerenke for
\textit{not strongly-normal maps}. In addition, we also study the \emph{Bloch mappings}, \emph{little-Bloch mappings}
and prove Zalcman\,--\,Pang type rescaling results for them. The framework is further
extended to \textit{strongly $\varphi$-normal mappings}, yielding a unified treatment across these settings.
\end{abstract}
\maketitle
\section{Introduction and Main Results}

The notion of normality in complex analysis originates from the study of convergence phenomena of 
holomorphic functions. Beginning with the foundational work of Weierstrass and leading to 
Montel’s theory of normal families, the idea of compactness through boundedness and equicontinuity 
became central to function theory. In the 1934, Yosida \cite{yosida-oacomf-1934} first shifted this 
perspective to individual meromorphic functions, a viewpoint later enhanced by Noshiro \cite{noshiro-ctttomfituc-1938}, which was formalized by Lehto and 
Virtanen \cite{lehto-bbanmf-1957} in 1957, who introduced the modern concept of normal meromorphic functions via conformal invariance and the spherical 
derivative. Noshiro \cite{noshiro-ctttomfituc-1938} provided a characterization of normal functions in terms of the spherical derivative. He showed that a 
function $f$ meromorphic in the unit disc $\mathbb{D}:=\{z\in\mathbb C:|z|<1\}$ is \emph{normal} if and only if \begin{equation*}
\sup_{z\in\mathbb{D}}(1-|z|^2)f^\#(z)<\infty ,   
\end{equation*} where $\displaystyle f^\#(z)=\frac{|f'(z)|}{1+|f(z)|^2}$ is the spherical derivative.\\

Later, Eremenko \cite{eremenko-nhcfprtps-2007} widened the scope of normality beyond meromorphic functions on $\mathbb D$ to holomorphic curves from
$\mathbb D$ into complex projective spaces ($\mathbb P^n(\mathbb C)$), thereby placing the theory within a broader geometric framework. Namely, a holomorphic curve $f:\mathbb D\to \mathbb P^n(\mathbb C)$ is \emph{normal} if and only if 
\begin{equation*}\sup_{z\in\mathbb D}(1-|z|^2)\|f'(z)\|_{FS}<\infty,\end{equation*}
where $\|f'(z)\|_{FS}$ is the Fubini\,--\,Study derivative which will be discussed in the preliminaries section.
\smallskip

Building upon this foundation, further refinements of the concept have been introduced. In particular, Chen and Gauthier \cite{chen-osnf-1996} proposed the 
notion of \textbf{strongly-normal} functions, requiring the sharper condition
\begin{equation*}(1-|z|^2)f^\#(z)\to 0,\ \text{as}\ z\to \partial \mathbb{D},\end{equation*}
where $\partial\mathbb D=\{z\in\mathbb C:|z|=1\}$. 
\smallskip

Following this, we now introduce the strong normality of holomorphic curves, which
is inspired by the definition of strongly-normal functions given by Chen and Gauthier \cite{chen-osnf-1996}.
\begin{definition}
A holomorphic curve $f:\mathbb{D}\to \mathbb P^n(\mathbb C)$ is \emph{strongly-normal} if 
\begin{equation*}(1-|z|^2)\|f'(z)\|_{FS}\to 0,\ \text{as}\ z\to\partial\mathbb{D}.\end{equation*}
\end{definition}

Recently, in 2019, Arbeláez et al. \cite{arbelaez-nhm-2019} introduced the concept of a normal function in the context of harmonic mappings. Subsequently, in 
2020, Jiang \cite{jiang-nolm-2020} extended this idea by defining a notion of normal function for logharmonic 
mappings. These definitions closely follow those introduced by Noshiro and Eremenko, with the spherical derivative adapted to the corresponding setting. \smallskip

To extend this notion beyond these, we briefly recall that a $\mathbb C$-valued mapping $f$ in $\mathbb D$ 
is harmonic if there exist holomorphic functions $g$, $h$ in $\mathbb D$ where $f(z)=h(z)+\overline{g(z)},\ z\in\mathbb D$. Furthermore, a $\mathbb C$-valued 
mapping $f$ defined in $\mathbb D$ is logharmonic if it satisfies the nonlinear elliptic partial differential equation
\begin{equation*}
\overline{f_{\bar z}} = a_f(z)\,\left(\frac{f_z}{f}\right){\overline{f}}, \tag{i} \label{NLEPDE}
\end{equation*}
where the function $a_f : \mathbb D \to \mathbb{C}$ is analytic and satisfies $|a_f(z)| < 1,\  \text{for all } z \in \mathbb D.$%
\smallskip
    
The above equation describes a class of mappings that generalizes analytic functions by allowing a 
controlled dependence on the conjugate variable. In particular, when $a_f \equiv 0$, 
the equation reduces to $f_{\bar z} = 0$ and hence $f$ is analytic in $\mathbb D$.
\smallskip

If $f$ is a non-vanishing logharmonic mapping in $\mathbb D$, then $f$ admits the representation
\begin{equation*}
f(z) = h(z)\,\overline{g(z)},\ z\in\mathbb D\tag{ii} \label{NVLHM}
\end{equation*}
where $h$ and $g$ are holomorphic functions in $\mathbb D$. \smallskip

If $f$ is a non-constant logharmonic mapping in $\mathbb D$ and vanishes only at the origin $z=0$, then $f$ can be written as
\begin{equation*}
f(z) = z^{m_0} |z|^{2\beta\,m_0} h(z)\overline{g(z)},\ z\in\mathbb D \tag{iii} \label{VLHM}
\end{equation*}
where $m_0$ is non-negative integer which represents the order of the zero at the origin, $\Re(\beta) > -\tfrac{1}{2}$,
$h$ and $g$ are holomorphic functions in $\mathbb D$, $g(0) = 1$ and $h(0) \neq 0$. 
It follows that $f(0) \neq 0$ if and only if $m_0 = 0$, in which case the representation reduces to 
(\ref{NVLHM}). Moreover, if $f$ is univalent in $\mathbb D$ and vanishes at the origin, then 
necessarily $m_0 = 1$ and thus $f$ assumes the form
$
f(z) = z\,|z|^{2\beta} h(z)\overline{g(z)}.
$\smallskip

Chen and Gauthier \cite{chen-osnf-1996} also introduced the notion of strong normality for harmonic mappings using an 
analogous formulation. We now proceed to define strong normality in logharmonic mappings. 
\begin{definition}
A logharmonic mapping $f=z|z|^{2\beta}h\overline{g}$ in $\mathbb{D}$ is \emph{strongly-normal} if \begin{equation*}(1-|z|^2)f^\diamondsuit(z)\to0,\ \text{as}\ z\to\partial\mathbb D,\end{equation*}
where \begin{equation*}f^\diamondsuit(z)=\dfrac{|f_z(z)|+|f_{\overline{z}}(z)|}{1+|f(z)|^2}.\end{equation*}
\end{definition}
It is interesting to note that the notion of a \emph{$\varphi$-normal function} was introduced by 
Aulaskari and R\"atty\"a \cite{aulaskari-pompnf-2011}, which defines a class,
significantly larger than the class of normal functions, in $\mathbb{D}$. Recently, Bohra et al. 
\cite{bohra-lfptfpnhm-2025} introduced the analogue of \emph{$\varphi$-normal function} 
for $\mathbb C$-valued \emph{harmonic mappings}. Note that the term $\varphi$ here denotes a $\mathbb R$-valued 
function whose growth is greater than that of growth of $1/(1-r^2)$ where $0\leq r<1$ and it admits a regularity condition close to $1$. \smallskip

Recall that an increasing function $\varphi : [0, 1)\to (0,\infty)$ is smoothly increasing if
\begin{equation*}\varphi(r)(1-r)\to \infty\ \text{ as }\ r \to 1\end{equation*}
and
\begin{equation*}R_a(z) := \frac{\varphi(|a + (z/\varphi(|a|))|)}
{\varphi(|a|)}\to 1\end{equation*} as $|a|\to 1$ uniformly on compact subsets of $\mathbb C$. With 
these preparations, we are now in a position to 
define \emph{$\varphi$-normal logharmonic mappings}.
\begin{definition}
A logharmonic mapping $f=z|z|^{2\beta}h\overline{g}$ in $\mathbb{D}$ is
$\varphi$-\emph{normal} if \begin{equation*}
\sup_{z\in\mathbb{D}}\frac{f^\diamondsuit(z)}{\varphi(|z|)}<\infty.
\end{equation*}
\end{definition}
In a similar way, we introduce \emph{strongly $\varphi$-normal harmonic mappings} and \emph{strongly $\varphi$-normal logharmonic mappings}.
\begin{definition}
A harmonic mapping $f=h+\overline{g}$ in $\mathbb D$ is \emph{strongly} $\varphi$-\emph{normal} if
\begin{equation*}\frac{f^\star(z)}{\varphi(|z|)}\to 0,\ \text{as}\ z\to \partial\mathbb D,\end{equation*}
 where \begin{equation*}f^\star(z)=\dfrac{|h^\prime(z)|+|g^\prime(z)|}{1+|f(z)|^2}.\end{equation*}
\end{definition}
\begin{definition}
A logharmonic mapping $f=z|z|^{2\beta}h\overline{g}$ in $\mathbb D$ is \emph{strongly} $\varphi$-\emph{normal} if
\begin{equation*}\frac{f^\diamondsuit(z)}{\varphi(|z|)}\to 0,\ \text{as}\ z\to \partial\mathbb D.\end{equation*}
\end{definition}
\smallskip
 Let us now focus on another class of functions named \emph{Bloch functions}. The concept of Bloch 
 functions was first studied in a systematic manner by Seidel and Walsh \cite{seidel-otdofaitucatrouaopv-1942} and was subsequently developed further by 
 several authors, including Pommerenke \cite{pommerenke-obf-1970}, Yamashita \cite{yamashita-cfftbb-1980} and Colonna \cite{colonna-banfatr-1989}. Colonna stated that a holomorphic function $f$ in $\mathbb D$ is \emph{Bloch} if 
 \begin{equation*}
     \sup_{z\in\mathbb D}(1-|z|^2)|f'(z)|<\infty.
 \end{equation*}
 
 In a similar way, Colonna formalized the definition in the harmonic setting 
 and introduced the class of \emph{Bloch harmonic mappings} 
 \cite{colonna-tbcobhm-1989}. We also recall the notion of \emph{little-Bloch functions}, introduced by Anderson et al. in \cite{anderson-obfanf-1974},
 which serves as a counterpart to strong normality by capturing vanishing boundary 
 behavior. That is, a holomorphic function $f$ in $\mathbb D$ is \emph{little-Bloch} if
 \begin{equation*}(1-|z|^2)|f'(z)|\to 0,\ \text{as}\ z\to\partial\mathbb D.\end{equation*}
 Motivated by this 
 perspective, Aljunaid and Colonna \cite{aljuain-cobtsohm-2019} introduced the notion of \emph{little-Bloch harmonic mappings} and Chang et al. 
 \cite{chang-csolmalbtc-2026} gave us the notion of \emph{Bloch logharmonic mappings} and \emph{little-Bloch logharmonic mappings}.\smallskip
 
We begin by recalling foundational rescaling principles for normality in various settings. In 
1973, Lohwater and Pommerenke \cite[Theorem 1]{lohwater-onmf-1973} established a theorem that characterizes non-normal meromorphic functions
in terms of their behavior under suitable rescalings near $\partial\mathbb D$. A cornerstone in 
this direction is the Lohwater\,–\,Pommerenke theorem for holomorphic curves in projective space, established by Hahn.
\begin{Result}[{\cite[Theorem 4]{hahn-hdgosctonmf-2007}}]\label{R:LPNHC} 
A holomorphic curve $f:\mathbb{D}\to \mathbb P^n(\mathbb C)$ is normal if and only if there does not
exist any sequence $\{z_{\nu}\}\subset \mathbb{D}$, a sequence $\{\rho_{\nu} \}\subset(0,\infty),\ \rho_{\nu} \searrow 0$
and $\zeta\in \mathbb C$ such that $g_{\nu}(\zeta)=f(z_{\nu}+\rho_{\nu} \zeta)$ converges locally uniformly on $\mathbb C$ to a non-constant 
holomorphic curve $g:\mathbb C\to \mathbb P^n{(\mathbb C)}$. 
\end{Result}
Motivated by this rescaling paradigm, one may ask how such phenomena behave under stronger notions of 
normality. In this direction, Aulaskari and Wulan developed an analogue tailored to \emph{strongly-normal meromorphic functions},
incorporating a quantitative restriction on the scaling.
\begin{Result} [{\cite[Theorem 1]{aulaskari-avotlptfsnf-2001}}] \label{R:LPMF}
A function $f$ meromorphic in $\mathbb{D}$ is strongly-normal function
if and only if there does not exist a constant $R>0$; a sequence $\{z_{\nu}\}\subset \mathbb{D}$ with $|z_{\nu}|\to1$; 
a sequence $\{\rho_{\nu} \}\subset(0,\infty)$ where $\displaystyle\frac{\rho_{\nu} }{1-|z_{\nu}|}<\frac{1}{2R}$ such that $\{f(z_{\nu}+\rho_{\nu} \zeta)\}$ converges uniformly on every 
compact subset to a non-constant meromorphic function in $|\zeta|<R$.
\end{Result}

Bohra et al. introduced an analogous framework for \emph{$\varphi$-normal harmonic mappings}, thereby
broadening the scope of the theory beyond normal harmonic mappings.
\begin{Result}[{\cite[Theorem 1.3]{bohra-lfptfpnhm-2025}}] \label{R:LPPNHM}
A non-constant mapping $f$ harmonic in $\mathbb D$ is not $\varphi$-normal if and only if there exist sequences $\{z_{\nu}\}$, $\{\rho_{\nu} \}$ where $\rho_{\nu} >0$ as $\nu\to\infty$ such that 
\begin{equation*}\lim_{\nu\to\infty} f\left(z_{\nu}+\frac{\rho_{\nu} \zeta}{\varphi(|z_{\nu}|)}\right)=g(\zeta)\end{equation*}
uniformly on compact subsets in $\mathbb C$, where $g$ is a non-constant harmonic mapping.
\end{Result}
Guided by the above results, we now formulate a partial Zalcman\,–\,Pang type generalization of Result \ref{R:LPMF}, beginning with meromorphic functions.
\begin{theorem} \label{T:FPZPSNMF}
    If a function $f$ meromorphic in $\mathbb D$ is not strongly-normal, then there exist a constant 
    $R>0$; a sequence $\{z_{\nu}\}\subset \mathbb D$ with $|z_{\nu}|\to1$; and a sequence 
    $\{\rho_{\nu} \}\subset (0,\infty)$ where $\displaystyle\frac{\rho_{\nu} }{1-|z_{\nu}|}<\frac{1}{2R}$ 
    such that \begin{equation*}g_{\nu}(\zeta)=\rho_{\nu} ^{\alpha}f(z_{\nu}+\rho_{\nu} \zeta)\end{equation*}
    converges locally uniformly to a non-constant meromorphic function $g(\zeta)$ in $|\zeta|<R$ 
    such that $g^\#(\zeta)\le g^\#(0)\neq 0$, where $\alpha\geq 0$. 
\end{theorem}
While the converse direction remains open in the meromorphic setting, further structure emerges when
restricted to holomorphic functions. This leads to the following analogue.
\begin{theorem} \label{T:FPZPSNHF}
 If a function $f$ holomorphic in $\mathbb D$ is not strongly-normal, then there exist a constant $R>0$; 
 a sequence $\{z_{\nu}\}\subset \mathbb D$ with $|z_{\nu}|\to1$; and a sequence $\{\rho_{\nu} \}\subset (0,\infty)$ where
 $\displaystyle\frac{\rho_{\nu} }{1-|z_{\nu}|}<\frac{1}{2R}$ such that \begin{equation*}g_{\nu}(\zeta)=\rho_{\nu} ^{\alpha}f(z_{\nu}+\rho_{\nu} \zeta)\end{equation*}
    converges locally uniformly to a non-constant holomorphic function $g(\zeta)$ in $|\zeta|<R$ 
    such that $g^\#(\zeta)\le g^\#(0)\neq 0$, where $\alpha\geq 0$.
\end{theorem}   
Moreover, in this holomorphic framework, a partial converse can be obtained through Bloch type considerations.
\begin{theorem} \label{T:CPZPSNHF}
    If there exist a constant $R>0$; a sequence $\{z_{\nu}\}\subset \mathbb D$ with 
    $|z_{\nu}|\to1$; and a sequence $\{\rho_{\nu} \}\subset (0,\infty)$ where 
    $\displaystyle\frac{\rho_{\nu} }{1-|z_{\nu}|}<\frac{1}{2R}$ such that \begin{equation*}g_{\nu}(\zeta)=\rho_{\nu} ^{\alpha}f(z_{\nu}+\rho_{\nu} \zeta)\end{equation*}
    converges locally uniformly to a non-constant holomorphic function $g(\zeta)$ in $|\zeta|<R$ 
    such that $g^\#(\zeta)\le g^\#(0)\neq 0$, where $\alpha\geq 0$, then $f$ is not a little-Bloch function.
\end{theorem}
Having established these connections in the holomorphic category, it is natural to revisit
harmonic mappings and investigate whether similar refined rescaling phenomena persist.
\begin{theorem} \label{T:FPZPSNHMF}
    If a harmonic mapping $f$ in $\mathbb D$ is not strongly-normal, then there exist a constant 
    $R>0$; a sequence $\{z_{\nu}\}\subset \mathbb D$ with $|z_{\nu}|\to1$; and a sequence $\{\rho_{\nu} \}\subset (0,\infty)$ where 
    $\displaystyle\frac{\rho_{\nu} }{1-|z_{\nu}|}<\frac{1}{2R}$ such that \begin{equation*}g_{\nu}(\zeta)=\rho_{\nu} ^{\alpha}f(z_{\nu}+\rho_{\nu} \zeta)\end{equation*}
    converges locally uniformly to a non-constant harmonic mapping $g(\zeta)$ in $|\zeta|<R$ such 
    that $g^\star(\zeta)\le g^\star(0)\neq 0$, where $\alpha\geq 0$.
\end{theorem}
A corresponding Bloch type implication also holds in this setting.
\begin{theorem} \label{T:CPZPSNHMF}  
    If there exist a constant $R>0$; a sequence $\{z_{\nu}\}\subset \mathbb D$ with 
    $|z_{\nu}|\to1$; and a sequence $\{\rho_{\nu} \}\subset (0,\infty)$ where $\displaystyle\frac{\rho_{\nu} }{1-|z_{\nu}|}<\frac{1}{2R}$ such that \begin{equation*}g_{\nu}(\zeta)=\rho_{\nu} ^{\alpha}f(z_{\nu}+\rho_{\nu} \zeta)\end{equation*}
    converges locally uniformly to a non-constant harmonic mapping $g(\zeta)$ in $|\zeta|<R$ such 
    that $g^\star(\zeta)\le g^\star(0)\neq 0$, where $\alpha\geq 0$, then $f$ is not little-Bloch harmonic mapping.
\end{theorem}  These observations naturally lead to an analogue of Result \ref{R:LPMF} for harmonic 
mappings, thus yielding a necessary and sufficient condition.
\begin{theorem} \label{T:LPSNHM}
A harmonic mapping $f$ in $\mathbb{D}$ is not strongly-normal
if and only if there exist a constant $R>0$; a 
sequence $\{z_{\nu}\}\subset \mathbb{D}$ with $|z_{\nu}|\to1$; and a sequence $\{\rho_{\nu} \}\subset(0,\infty)$ satisfying 
$\displaystyle\frac{\rho_{\nu} }{1-|z_{\nu}|}<\frac{1}{2R}$ such that $\{f(z_{\nu}+\rho_{\nu} \zeta)\}$ converges locally uniformly to a non-constant harmonic mapping in $|\zeta|<R$.
\end{theorem}
We now turn our attention to logharmonic mappings, where a similar rescaling mechanism can be established.
\begin{theorem} \label{T:FPZPSNLHMF}
    If a logharmonic mapping $f$ in $\mathbb D$ is not strongly-normal, then there exist a constant $R>0$, a sequence $\{z_{\nu}\}\subset \mathbb D$
    with $|z_{\nu}|\to1$ and a sequence $\{\rho_{\nu} \}\subset (0,\infty)$ where $\displaystyle\frac{\rho_{\nu} }{1-|z_{\nu}|}<\frac{1}{2R}$ such that \begin{equation*}g_{\nu}(\zeta)=\rho_{\nu} ^{\alpha}f(z_{\nu}+\rho_{\nu} \zeta)\end{equation*}
    converges locally uniformly to a non-constant logharmonic mapping $g(\zeta)$ in $|\zeta|<R$ 
    such that $g^\diamondsuit(\zeta)\le g^\diamondsuit(0)\neq0$, where $\alpha\geq 0$.
\end{theorem}
An analogous Bloch type implication follows in this setting.
\begin{theorem} \label{T:CPZPSNLHMF}
    If there exist a constant $R>0$; a sequence $\{z_{\nu}\}\subset \mathbb D$ with 
    $|z_{\nu}|\to1$; and a sequence $\{\rho_{\nu} \}\subset (0,\infty)$ where 
    $\displaystyle\frac{\rho_{\nu} }{1-|z_{\nu}|}<\frac{1}{2R}$ such that \begin{equation*}g_{\nu}(\zeta)=\rho_{\nu} ^{\alpha}f(z_{\nu}+\rho_{\nu} \zeta)\end{equation*}
    converges locally uniformly to a non-constant logharmonic mapping $g(\zeta)$ in $|\zeta|<R$ 
    such that $g^\diamondsuit(\zeta)\le g^\diamondsuit(0)\neq0$, where $\alpha\geq 0$, then $f$ is not little-Bloch logharmonic mapping.
\end{theorem}
Moreover, we obtain a characterization analogous to that of Lohwater and Pommerenke.
\begin{theorem} \label{T:LPSNLHM}
A logharmonic mapping $f$ in $\mathbb{D}$ is not strongly-normal
if and only if there exist a constant $R>0$; a sequence $\{z_{\nu}\}\subset \mathbb{D}$ with 
$|z_{\nu}|\to1$; and a sequence $\{\rho_{\nu} \}\subset(0,\infty)$ satisfying 
$\displaystyle\frac{\rho_{\nu} }{1-|z_{\nu}|}<\frac{1}{2R}$ such that $\{f(z_{\nu}+\rho_{\nu} \zeta)\}$ 
converges locally uniformly to a non-constant logharmonic mapping in $|\zeta|<R$.
\end{theorem}
Returning to the projective setting, where Pang’s 
rescaling lemma interacts naturally with homogeneous coordinates, we obtain the following extension, which may be viewed as a unified analogue of Results 
\ref{R:LPMF} and \ref{R:LPNHC}, noting that a Lohwater\,–\,Pommerenke type theorem for strongly-normal holomorphic curves in complex projective space 
has not previously appeared in the literature.
\begin{theorem} \label{T:LPSNHCPNC}
A holomorphic curve $f:\mathbb{D}\to \mathbb P^n(\mathbb C)$ is not strongly-normal if and only if
there exist a constant $R>0$, a sequence $\{z_{\nu}\}\subset \mathbb{D}$ with $|z_{\nu}|\to1$
and a sequence $\{\rho_{\nu} \}\subset (0,\infty)$ satisfying $\displaystyle\frac{\rho_{\nu} }{1-|z_{\nu}|}<\frac{1}{2R}$ such that 
${f(z_{\nu}+\rho_{\nu} \zeta)}$ converges locally uniformly in $|\zeta|<R$ to a non-constant holomorphic curve $g:\{\zeta\in\mathbb C\ |\ |\zeta|<R\}\to\mathbb P^n(\mathbb C)$.
\end{theorem}
Finally, we turn to strongly $\varphi$-normal harmonic mappings and develop the Lohwater\,--
\,Pommerenke Theorem for strongly $\varphi$-normal harmonic mappings.
\begin{theorem} \label{T:LPSPNHM}
A harmonic mapping $f$ in $\mathbb D$ is not strongly $\varphi$-normal if and only if there exist a 
constant $R>0$; a sequence $\{z_{\nu}\}\subset\mathbb D$ with $|z_{\nu}|\to1$; and a sequence $\{\rho_{\nu} \}\subset(0,\infty)$ where 
$\displaystyle\frac{\rho_{\nu} }{1-|z_{\nu}|}<\frac{1}{2R}$ such that $f\left(z_{\nu}+\frac{\rho_{\nu} \zeta}{\varphi(|z_{\nu}|)}\right)$ converges locally 
uniformly to a non-constant harmonic mapping $g$ in $|\zeta|<R$.
\end{theorem}
We establish a corresponding characterization in the logharmonic setting.
\begin{theorem} \label{T:LPSPNLHM}
A logharmonic mapping $f$ in $\mathbb D$ is not strongly $\varphi$-normal if and only if there exist 
a constant $R>0$; a sequence $\{z_{\nu}\}\subset\mathbb D$ with $|z_{\nu}|\to1$; and 
a sequence $\{\rho_{\nu} \}\subset(0,\infty)$ where $\displaystyle\frac{\rho_{\nu} }{1-|z_{\nu}|}<\frac{1}{2R}$ such that $f\left(z_{\nu}+\frac{\rho_{\nu} \zeta}{\varphi(|z_{\nu}|)}\right)$ converges locally 
uniformly 
to a non-constant logharmonic mapping $g$ in $|\zeta|<R$.
\end{theorem}
\section{Preliminaries and Essential Lemmas}
In this section, we introduce the basic concepts and notation that will be used throughout the paper. For 
a comprehensive background, we refer the reader to \cite{chang-csolmalbtc-2026, aljuain-cobtsohm-2019, anderson-obfanf-1974, arbelaez-nhm-2019,aulaskari-avotlptfsnf-2001,bohra-lfptfpnhm-2025,chen-osnf-1996,colonna-banfatr-1989,colonna-tbcobhm-1989,eremenko-nhcfprtps-2007,eremenko-bcoh-2008,jiang-nolm-2020,pang-aeostfnf-2015,quang-wnaqfohc-2018}.\smallskip

The complex $n$-projective space is denoted by  $\mathbb{P}^n(\mathbb{C})$; fixing the homogeneous coordinates
$[z_0:z_1:\cdots:z_n]$ in $\mathbb P^n(\mathbb C)$ and assuming that $f : \mathbb{C} \to \mathbb{P}^n(\mathbb{C})$ be a holomorphic curve represented by 
$f(z) = [f_0(z) : f_1(z) : \dots : f_n(z)]$, we now go through the terms of the Fubini\,–\,Study distance, the Fubini\,–\,Study norm and the Fubini\,–
\,Study derivative \cite{eremenko-nhcfprtps-2007, quang-wnaqfohc-2018}.

\begin{enumerate}
    \item The \emph{Fubini\,--\,Study distance} between any two points in $\mathbb P^n(\mathbb C)$ is defined as
\begin{equation*}
d_{FS}(p,q) = \frac{\sum_{k,l=0}^n |p_k q_l - p_l q_k|}{\sqrt{\sum_{k=0}^n |p_k|^2} \, \sqrt{\sum_{l=0}^n |q_l|^2}},
\end{equation*}
for $p = [p_0 : p_1 : \dots : p_n]$ and $q = [q_0 : q_1 : \dots : q_n]$ in $\mathbb{P}^n(\mathbb{C})$. \smallskip
    \item The \emph{Fubini\,--\,Study norm} of $f$ at $z$ is defined as \begin{equation*}
\|f(z)\|_{FS} = \left( \sum_{k=0}^n |f_k(z)|^2 \right)^{1/2}.
\end{equation*}
    \item The \emph{Fubini\,--\,Study derivative} of $f$ at $z$ is defined as \begin{equation*}
\|f'(z)\|_{FS} = 
\frac{\left( \sum_{0 \le l_1 < l_2 \le n} 
\left| f_{l_1}'(z) f_{l_2}(z) - f_{l_1}(z) f_{l_2}'(z) \right|^2 \right)^{1/2}}
{\sum_{k=0}^n |f_k(z)|^2}.
\end{equation*}
\end{enumerate}

We provide a set of lemmas that will be used before moving on to the proofs of the key theorems. \smallskip

The lemma serves as a fundamental rescaling criterion for non-normality for a family of holomorphic 
functions. By following analogous arguments, the same methodology extends naturally to broader classes of functions. In particular, Hua’s \cite{hua-anatnc-1995} approach adapts to meromorphic function without 
essential modification; similar ideas have also been used in the harmonic setting by Bharti and Thin \cite{bharti-ronhapnhm-2025}.\smallskip
\begin{lemma}[{\cite[Lemma 6]{hua-anatnc-1995}}] \label{L-Hua} 
Given $\mathcal{F}$ be a family of holomorphic functions on $\mathbb D$. Then $\mathcal F$ is non-normal at $z=0$ if and only if given any 
$\alpha\in(-1,\infty)$, there is a sequence $\{f_{\nu}\}\subset\mathcal F$, a sequence $\{z_{\nu}\}\subset \mathbb D$ and a sequence of 
positive numbers $\rho_{\nu} \searrow 0$ such that $g_{\nu}(\zeta)=\rho_{\nu} ^{\alpha}f_{\nu}(z_{\nu}+\rho_{\nu} \zeta)$ converges locally 
uniformly with respect to the spherical metric to a non-constant entire function $g(\zeta)$.
\end{lemma}

Motivated by this lemma, we formulate the corresponding version for a single meromorphic function.  
\begin{lemma} \label{L-nnmf}
If a meromorphic function $f$ on $\mathbb D$ is non-normal, then given any $\alpha\in[0,\infty)$, there is a sequence $\{z_{\nu}\}\subset \mathbb D$ and a 
sequence of positive numbers $\rho_{\nu} \searrow 0$ such that $g_{\nu}(\zeta)=\rho_{\nu} ^{\alpha}f_{\nu}(z_{\nu}+\rho_{\nu} \zeta)$ converges locally 
uniformly with respect to the spherical metric to a non-constant meromorphic function $g(\zeta)$.
\end{lemma}
\begin{proof}
The proof of this lemma follows along the same lines as that of Bharti\,--\,Thin 
\cite[Theorem 1.2]{bharti-ronhapnhm-2025} and Ahamad\,--\,Mandal \cite[Theorem 2.2]{ahamed-ncflm-2025}.
\end{proof}
We now turn to the harmonic setting, where an analogous rescaling characterization of non-normality has been established.
\begin{lemma}[{\cite[Theorem 1.2]{bharti-ronhapnhm-2025}}]\label{L-nnhm} 
A non-constant mapping $f$ harmonic in $\mathbb D$ is non-normal if and only if there exist sequence of
points $\{z_{\nu}\}\subset\mathbb D$ and of real numbers $\{\rho_{\nu} \}\subset(0,\infty)$ with 
$\rho_{\nu} \searrow 0$ as $\nu\to\infty$ such that the functions $g(\zeta)=\rho_{\nu} ^{\alpha}f(z_{\nu}+\rho_{\nu} \zeta)$ converges 
locally uniformly to a non-constant harmonic mapping $g(\zeta)$ where $g^{\star}(\zeta)\leq g^\star(0)=1$, where $\alpha\in(-1,\infty)$. 
\end{lemma}
The above result naturally extends to the logharmonic framework, where a similar rescaling principle characterizes non-normality.
\begin{lemma}[{\cite[Theorem 2.2]{ahamed-ncflm-2025}}] \label{L-nnlhm} 
A non-constant mapping $f$ logharmonic in $\mathbb D$ is non-normal if and only if there exist sequence of points $\{z_{\nu}\}\subset\mathbb D$ and of real 
numbers $\{\rho_{\nu} \}\subset(0,\infty)$ with $\rho_{\nu} \searrow 0$ as $\nu\to\infty$ such that the functions $g(\zeta)=\rho_{\nu} ^{\alpha}f(z_{\nu}+\rho_{\nu} \zeta)$ converges 
locally uniformly to a non-constant logharmonic mapping $g(\zeta)$ whre $g^{\diamondsuit}(\zeta)\leq g^\diamondsuit(0)=1$, where $\alpha\in(-1,\infty)$. 
\end{lemma}
It is worth noting that, in the above result due to Ahamed and Mandal, the converse direction is not 
correctly established. For the purposes of this article, we therefore restrict our attention to the 
forward implication. In the same way, throughout our subsequent results, we will consider only with the forward direction of Lemmas \ref{L-nnhm} and 
\ref{L-nnlhm} and confine ourselves to the range $\alpha\geq 0$.\smallskip

We now proceed to a corresponding characterization in the setting of $\varphi$-normal logharmonic mappings.
\begin{lemma} \label{L-phinnlhm} 
A non-constant mapping $f$ logharmonic in $\mathbb D$ is not $\varphi$-normal if and only if there exist
sequences $\{z_{\nu}\}$, $\{\rho_{\nu} \}$ where $\rho_{\nu} >0 $ as $\nu\to\infty$ such that 
\begin{equation*}\lim_{\nu\to\infty} f\left(z_{\nu}+\frac{\rho_{\nu} \zeta}{\varphi(|z_{\nu}|)}\right)=g(\zeta)\end{equation*}
uniformly on every compact subset in $\mathbb C$, where $g$ is a non-constant logharmonic mapping.
\end{lemma}
\begin{proof} The proof of this lemma follows mutatis mutandis from the proof of \cite[Theorem 1.3]{bohra-lfptfpnhm-2025} established by Bohra et al.\end{proof}
\section{Proof of the Main Results}
\begin{proof}[{\bf Proof of Theorem~\ref{T:FPZPSNMF}}]
Suppose $f$ is not strongly-normal. Then there exist a positive constant $c$ and a sequence $\{\tilde{z}_{\nu}\} \subset \mathbb D$
with $\tilde{z}_{\nu} \to 1$ as $\nu \to \infty$ such that
\begin{equation}
(1-|\tilde{z}_{\nu}|^{2}) f^{\#}(\tilde{z}_{\nu})\geq c \quad \text{for }\ \nu\in\mathbb N. \tag{\ref*{T:FPZPSNMF}.1}
\end{equation}
If $f$ is not a normal function, the result follows from Lemma \ref{L-nnmf}. So we assume that $f$ is normal function and assume that 
$|\tilde{z}_{\nu}|>1/2$ for all $\nu\in \mathbb N$. \smallskip

Let $s_{\nu}=2|\tilde{z}_{\nu}|/(1+|\tilde{z}_{\nu}|)$ and $s_{\nu}'=|\tilde{z}_{\nu}|/(2-|\tilde{z}_{\nu}|)$ for all $\nu\in\mathbb N$. Then 
$1/3<s_{\nu}'<|\tilde{z}_{\nu}|<s_{\nu}<1$ and $s_{\nu}\to 1,\ s_{\nu}'\to 1$ as $\nu\to \infty$ and there exist a constant $c$ such that
\begin{align*}
\left(1-\tfrac{|\tilde{z}_{\nu}|}{s_{\nu}}\right)^{1/2} \left(1-\tfrac{s_{\nu}'}{|\tilde{z}_{\nu}|}\right)^{1/2} f^{\#}(\tilde{z}_{\nu})&= \left(\frac{1-|\tilde{z}_{\nu}|}{2}\right)^{1/2} \left(\frac{1-|\tilde{z}_{\nu}|}{2-|\tilde{z}_{\nu}|}\right)^{1/2} f^\#(\tilde{z}_{\nu}) \\
&\geq \frac{1}{2}(1-|\tilde{z}_{\nu}|)f^\#(\tilde{z}_{\nu})\\ &\geq \frac{c}{4}>0 \quad \text{for } \nu \in\mathbb N .
\tag{\ref*{T:FPZPSNMF}.2}\label{SNMF.2}
\end{align*}
Also we have
\begin{equation*}s_{\nu}-s_{\nu}'=\frac{2|\tilde{z}_{\nu}|}{1+|\tilde{z}_{\nu}|}-\frac{|\tilde{z}_{\nu}|}{2-|\tilde{z}_{\nu}|}=\frac{3|\tilde{z}_{\nu}|(1-|\tilde{z}_{\nu}|)}{(1+|\tilde{z}_{\nu}|)(2-|\tilde{z}_{\nu}|)}\end{equation*} and \begin{equation*}1-s_{\nu}=\frac{1-|\tilde{z}_{\nu}|}{1+|\tilde{z}_{\nu}|}.\end{equation*}
Thus
\begin{align*}
3(1-s_{\nu})\geq\frac{3|\tilde{z}_{\nu}|(1-|\tilde{z}_{\nu}|)}{1+|\tilde{z}_{\nu}|}\geq \frac{3|\tilde{z}_{\nu}|(1-|\tilde{z}_{\nu}|)}{(1+|\tilde{z}_{\nu}|)(2-|\tilde{z}_{\nu}|)}=s_{\nu}-s_{\nu}'.\tag{\ref*{T:FPZPSNMF}.2.1}\label{SNMF.2.1} \end{align*} 
Define
\begin{equation*}
F_{\nu}(t,z)
:=
\dfrac{\left(1-\frac{|z|}{s_{\nu}}\right)^{(1+\alpha)/2} \left(1-\frac{s_{\nu}'}{|z|}\right)^{(1+\alpha)/2}
t^{1+\alpha} (1+|f(z)|^{2}) f^{\#}(z)}
{1+\left(1-\frac{|z|}{s_{\nu}}\right)^{\alpha} \left(1-\frac{s_{\nu}'}{|z|}\right)^{\alpha} t^{2\alpha} |f(z)|^{2}},
\end{equation*}
where $0<t\le1$ and $s_{\nu}'<|z|<s_{\nu}$. Evidently, the functions $F_{\nu}(t,z)$ are continuous
in $(0,1]\times\{s_{\nu}'<|z|<s_{\nu}\}$. Also, since $\alpha\geq0$ and $f$ is meromorphic, it follows that
\begin{equation}
\lim_{t\to 0} F_{\nu}(t,z)=0 .
\tag{\ref*{T:FPZPSNMF}.3} \label{SNMF.3}
\end{equation}

We claim that
\begin{equation}
F_{\nu}(t,z)
\ge
t^{1+\alpha}
\left(1-\tfrac{|z|}{s_{\nu}}\right)^{(1+\alpha)/2} \left(1-\tfrac{s_{\nu}'}{|z|}\right)^{(1+\alpha)/2}
f^{\#}(z).
\tag{\ref*{T:FPZPSNMF}.4} \label{SNMF.4}
\end{equation}

As $\alpha\geq0$, then using
\begin{equation*}
\left(1-\tfrac{|z|}{s_{\nu}}\right)^{\alpha} \left(1-\tfrac{s_{\nu}'}{|z|}\right)^{\alpha} t^{2\alpha} \le 1,
\end{equation*}
we obtain
\begin{align*}    
F_{\nu}(t,z)
&\ge
\frac{\left(1-\frac{|z|}{s_{\nu}}\right)^{(1+\alpha)/2} \left(1-\frac{s_{\nu}'}{|z|}\right)^{(1+\alpha)/2}
t^{1+\alpha}(1+|f(z)|^{2})f^{\#}(z)}{1+|f(z)|^{2}}
\\ 
&=
t^{1+\alpha}
\left(1-\tfrac{|z|}{s_{\nu}}\right)^{(1+\alpha)/2} \left(1-\tfrac{s_{\nu}'}{|z|}\right)^{(1+\alpha)/2}
f^{\#}(z).
\end{align*}

From (\ref{SNMF.2}) and (\ref{SNMF.4}), we have
\begin{equation*}
F_{\nu}(1,\tilde{z}_{\nu})
\ge
\left(1-\tfrac{|\tilde{z}_{\nu}|}{s_{\nu}}\right)^{(1+\alpha)/2} \left(1-\tfrac{s_{\nu}'}{|\tilde{z}_{\nu}|}\right)^{(1+\alpha)/2}
f^{\#}(\tilde{z}_{\nu})
\geq c_0 \quad \text{for } \nu \in\mathbb N,
\end{equation*} where $c_0>0$.
for sufficiently large $\nu$, $F_{\nu}(1,\tilde{z}_{\nu})\geq c_1$, where $0<c_1<c_0$ and hence
\begin{equation*}
\sup_{s_{\nu}'<|z|<s_{\nu}} F_{\nu}(1,z)\geq F_{\nu}(1,\tilde{z}_{\nu}) \geq c_1.
\end{equation*}
Also, from (\ref{SNMF.3}), it follows that
\begin{equation*}
\sup_{s_{\nu}'<|z|<s_{\nu}} F_{\nu}(t,z) < c_1
\quad \text{for sufficiently small } t.
\end{equation*}
Thus, by the intermediate value theorem, there exist $t_{\nu}\in(0,1)$ and $z_{\nu}$
with $s_{\nu}'<|z_{\nu}|<s_{\nu}$ such that
\begin{equation}
\sup_{s_{\nu}'<|z|<s_{\nu}} F_{\nu}(t_{\nu},z) = F_{\nu}(t_{\nu},z_{\nu})=c_1 .
\tag{\ref*{T:FPZPSNMF}.5} \label{SNMF.5}
\end{equation}
From (\ref{SNMF.4}) and (\ref{SNMF.5}), we get
\begin{align*}
c_1&= F_{\nu}(t_{\nu},z_{\nu})
\\&\ge F_{\nu}(t_{\nu},\tilde{z}_{\nu})\\ &\ge
t_{\nu}^{1+\alpha}
\left(1-\tfrac{|\tilde{z}_{\nu}|}{s_{\nu}}\right)^{(1+\alpha)/2} \left(1-\tfrac{s_{\nu}'}{|\tilde{z}_{\nu}|}\right)^{(1+\alpha)/2}
f^{\#}(\tilde{z}_{\nu})
\\&=
t_{\nu}^{1+\alpha} F_{\nu}(1,\tilde{z}_{\nu}).
\end{align*}
Since $F_{\nu}(1,\tilde{z}_{\nu})\geq c_0$, it follows that as $\nu\to\infty$, $t_{\nu}\leq c_1/c_0$. Put
\begin{equation*}
\rho_{\nu} 
:=
\left(1-\tfrac{|\tilde{z}_{\nu}|}{s_{\nu}}\right)^{1/2} \left(1-\tfrac{s_{\nu}'}{|\tilde{z}_{\nu}|}\right)^{1/2} t_{\nu}.
\end{equation*}
From (\ref{SNMF.2.1}), we get
\begin{equation*}1-|z_{\nu}|>1-s_{\nu}\geq \frac{1}{3}(s_{\nu}-s_{\nu}')\geq \frac{1}{3}(|z_{\nu}|-s_{\nu}'),\ \ \nu\in\mathbb N.\end{equation*}
Then
\begin{align*}
\frac{\rho_{\nu} }{1-|z_{\nu}|}
&\le
\frac{\rho_{\nu} }{(s_{\nu}-|z_{\nu}|)^{1/2}\left(\frac{1}{3}(|z_{\nu}|-s_{\nu}')\right)^{1/2}}\\
&=
\frac{t_{\nu}}{s_{\nu}^{1/2}|z_{\nu}|^{1/2}}\\ &<\frac{12\sqrt{3}c_1}{c_0}:=\frac{1}{2R}
\quad \text{for } \nu\in \mathbb N .
\tag{\ref*{T:FPZPSNMF}.6} \label{SNMF.6}
\end{align*}
Thus the function
$
g_{\nu}(\zeta)
=
\rho_{\nu} ^{\alpha} f(z_{\nu}+\rho_{\nu} \zeta)
$
is defined for $|\zeta|<R$.
Carrying out the computation, we obtain
\begin{equation}
g_{\nu}^{\#}(\zeta)
=
\frac{\rho_{\nu} ^{1+\alpha}(1+|f(z_{\nu}+\rho_{\nu} \zeta)|^{2})f^{\#}(z_{\nu}+\rho_{\nu} \zeta)}
{1+\rho_{\nu} ^{2\alpha}|f(z_{\nu}+\rho_{\nu} \zeta)|^{2}} .
\tag{\ref*{T:FPZPSNMF}.7} \label{SNMF.7}
\end{equation}
In particular,
\begin{align*}
g_{\nu}^{\#}(0)
=&~
\frac{\rho_{\nu} ^{1+\alpha}(1+|f(z_{\nu})|^{2})f^{\#}(z_{\nu})}
{1+\rho_{\nu} ^{2\alpha}|f(z_{\nu})|^{2}}\\
=&~ \dfrac{\left(1-\tfrac{|\tilde{z}_{\nu}|}{s_{\nu}}\right)^{(1+\alpha)/2} \left(1-\tfrac{s_{\nu}'}{|\tilde{z}_{\nu}|}\right)^{(1+\alpha)/2} t_{\nu}^{1+\alpha}(1+|f(z_{\nu})|^{2})f^{\#}(z_{\nu})}{1+\left(1-\tfrac{|\tilde{z}_{\nu}|}{s_{\nu}}\right)^{\alpha} \left(1-\tfrac{s_{\nu}'}{|\tilde{z}_{\nu}|}\right)^{\alpha} t_{\nu}^{2\alpha}|f(z_{\nu})|^2}\\
=&~F_{\nu}(t_{\nu},z_{\nu})=c_1.
\tag{\ref*{T:FPZPSNMF}.8} \label{SNMF.8}
\end{align*}

We now need to show that the sequence $\{g_{\nu}(\zeta)\}$ is normal. 
\smallskip

Since $\left[\left(1-\frac{|z_{\nu}|}{s_{\nu}}\right)^{1/2}\left(1-\frac{s_{\nu}'}{|z_{\nu}|}\right)^{1/2}\right]/\left[\left(1-\frac{|z_{\nu}+\rho_{\nu} \zeta|}{s_{\nu}}\right)^{1/2}\left(1-\frac{s_{\nu}'}{|z_{\nu}+\rho_{\nu} \zeta|}\right)^{1/2}\right]\to 1$ as $\nu\to \infty$, there exist $\epsilon_{\nu}>0$ with $\epsilon_\nu\to 0$ as $\nu\to \infty$ such that 
\begin{align*} \label{SNMF.9}
    &(1-\epsilon_{\nu})\left(1-\tfrac{|z_{\nu}+\rho_{\nu} \zeta|}{s_{\nu}}\right)^{1/2}\left(1-\tfrac{s_{\nu}'}{|z_{\nu}+\rho_{\nu} \zeta|}\right)^{1/2}t_{\nu}\\ &\leq \rho_{\nu} \leq (1+\epsilon_{\nu})\left(1-\tfrac{|z_{\nu}+\rho_{\nu} \zeta|}{s_{\nu}}\right)^{1/2}\left(1-\tfrac{s_{\nu}'}{|z_{\nu}+\rho_{\nu} \zeta|}\right)^{1/2}t_{\nu} \tag{\ref*{T:FPZPSNMF}.9}
\end{align*} 
Take $\Upsilon_{\nu}=\left(1-\tfrac{|z_{\nu}+\rho_{\nu} \zeta|}{s_{\nu}}\right)^{1/2} \left(1-\tfrac{s_{\nu}'}{|z_{\nu}+\rho_{\nu} \zeta|}\right)^{1/2} t_{\nu}$.
Therefore, from (\ref{SNMF.5}), (\ref{SNMF.7}) and (\ref{SNMF.9}), we obtain 
\begin{align*}
g^\#_n(\zeta)&\le  \dfrac{(1+\epsilon_{\nu})^{1+\alpha}\Upsilon_{\nu}^{1+\alpha}(1+|f(z_{\nu}+\rho_{\nu} \zeta)|)f^{\#}(z_{\nu}+\rho_{\nu} \zeta)}{1+(1-\epsilon_{\nu})^{2\alpha}\Upsilon_{\nu}^{2\alpha}|f(z_{\nu}+\rho_{\nu} \zeta)|^2}\\
&\le  \left[\dfrac{(1+\epsilon_{\nu})^{1+\alpha}\Upsilon_{\nu}^{1+\alpha}(1+|f(z_{\nu}+\rho_{\nu} \zeta)|)f^{\#}(z_{\nu}+\rho_{\nu} \zeta)}{1+\Upsilon_{\nu}^{2\alpha}|f(z_{\nu}+\rho_{\nu} \zeta)|^2}\right]\\
&\ \ \ ~~  \left[\dfrac{1+\Upsilon_{\nu}^{2\alpha}|f(z_{\nu}+\rho_{\nu} \zeta)|^2}{1+(1-\epsilon_{\nu})^{2\alpha}\Upsilon_{\nu}^{2\alpha}|f(z_{\nu}+\rho_{\nu} \zeta)|^2}\right]\\
&\leq (1+\epsilon_{\nu})^{1+\alpha}F_{\nu}(t_{\nu},z_{\nu})\left[\dfrac{1+\Upsilon_{\nu}^{2\alpha}|f(z_{\nu}+\rho_{\nu} \zeta)|^2}{1+(1-\epsilon_{\nu})^{2\alpha}\Upsilon_{\nu}^{2\alpha}|f(z_{\nu}+\rho_{\nu} \zeta)|^2}\right]\\
&\leq (1+\epsilon_{\nu})^{1+\alpha} F_{\nu}(t_{\nu},z_{\nu})\left[1+(1-(1-\epsilon_{\nu})^{2\alpha})\Upsilon_{\nu}^{2\alpha}|f(z_{\nu}+\rho_{\nu} \zeta)|^2\right]\\
&\to c_1\ \text{ as }\ \nu\to\infty.
\tag{\ref*{T:FPZPSNMF}.10} \label{SNMF.10}
\end{align*}

Hence, by Marty's criterion \cite[p. 226, Theorem 17]{ahlfors-ca-1968}, $\{g_{\nu}\}$ is a normal family 
in $|\zeta|<R$ and therefore converges locally
uniformly in $|\zeta|<R$ to a meromorphic mapping $g$. From (\ref{SNMF.8}) and (\ref{SNMF.10}),
\begin{equation*}
g^{\#}(\zeta) \le g^{\#}(0)=c_1 \neq 0,
\end{equation*}
so $g$ is non-constant.
\end{proof}
\begin{proof}[{\bf Proof of Theorem~\ref{T:FPZPSNHF}}]
The proof follows mutatis mutandis from the previous theorem. 
\end{proof}
\begin{proof}[{\bf Proof of Theorem~\ref{T:CPZPSNHF}}]
Suppose that 
$
F_{\nu}(\zeta) =\rho_{\nu} ^{\alpha} f(z_{\nu} + \rho_{\nu}  \zeta)
$
converges locally uniformly to $F(\zeta)$ in $|\zeta| < R$, where $F(\zeta)$ is a non-constant holomorphic function and $|z_{\nu}| \to 1$ with 
\begin{equation*}
\frac{\rho_{\nu} }{1 - |z_{\nu}|} < \frac{1}{2R}
\quad \text{for all } \nu \in \mathbb{N}.
\end{equation*}
Choose $\zeta_0$, $|\zeta_0| < R$, such that $F^\#(\zeta_0) > 0$.
The proof will be completed by contradiction.
As 
\begin{align*}
    F_{\nu}^\#(\zeta_0)=&~\frac{\rho_{\nu} ^{1+\alpha}(1+|f(z_{\nu}+\rho_{\nu} \zeta_0|^2)f^\#(z_{\nu}+\rho_{\nu} \zeta_0)}{1+\rho_{\nu} ^{2\alpha}|f(z_{\nu}+\rho_{\nu} \zeta_0|^2}\\
    \le&~\rho_{\nu} ^{1+\alpha}f^\#(z_{\nu}+\rho_{\nu} \zeta_0)\left(\frac{1+|f(z_{\nu}+\rho_{\nu} \zeta_0|^2}{1+\rho_{\nu} ^{2\alpha}|f(z_{\nu}+\rho_{\nu} \zeta_0|^2}\right)\\
    \le&~\rho_{\nu} ^{1+\alpha}f^\#(z_{\nu}+\rho_{\nu} \zeta_0)\left(1+|f(z_{\nu}+\rho_{\nu} \zeta_0)|^2\right)\\
    \le&~ \frac{\rho_{\nu} ^{\alpha+1}\left(1+|f(z_{\nu}+\rho_{\nu} \zeta_0)|^2\right)}{1-|z_{\nu}|-\rho_{\nu} |\zeta_0|}\left(1-|z_{\nu}+\rho_{\nu} \zeta_0|^2 \right)f^\#(z_{\nu}+\rho_{\nu} \zeta_0)\\
    \le&~ \frac{1}{R}\rho_{\nu} ^{\alpha}\left(1+|f(z_{\nu}+\rho_{\nu} \zeta_0)|^2\right)\left(1-|z_{\nu}+\rho_{\nu} \zeta_0|^2 \right)f^\#(z_{\nu}+\rho_{\nu} \zeta_0)\\
    \le&~ \frac{1}{R}\rho_{\nu} ^{\alpha}\left(1-|z_{\nu}+\rho_{\nu} \zeta_0|^2 \right)|f'(z_{\nu}+\rho_{\nu} \zeta_0)|.
\end{align*}
As $f$ is little-Bloch function, so $F_{\nu}^{\#}(\zeta_0)\to 0$. So $F^\#(\zeta_0)=0$ and so $F$ is 
constant which is a contradiction. Thus $f$ is not little-Bloch function.
\end{proof}
\begin{proof}[{\bf Proofs of Theorem~\ref{T:FPZPSNHMF} and \ref{T:FPZPSNLHMF}}]
The proofs proceed mutatis mutandis as in the proof of Theorem~\ref{T:FPZPSNHF}.  
\end{proof}
\begin{proof}[{\bf Proofs of Theorem~\ref{T:CPZPSNHMF} and \ref{T:CPZPSNLHMF}}]
The proofs are obtained mutatis mutandis from the 
proof of Theorem~\ref{T:CPZPSNHF}.
\end{proof}
\begin{proof}[{\bf Proofs of Theorem~\ref{T:LPSNHM}, \ref{T:LPSNLHM} and \ref{T:LPSNHCPNC}}]
The proofs follow by a straightforward adaptation of 
the proof of \cite[Theorem 1]{aulaskari-avotlptfsnf-2001} due to Aulaskari\,--\,Wulan.
\end{proof}
\begin{proof}[{\bf Proof of Theorem~\ref{T:LPSPNHM}}]
Suppose $f$ is not a strongly $\varphi$-normal harmonic mapping, then there exist $c>0$, a sequence 
$\{\tilde{z}_{\nu}\}$ of points in $\mathbb D$ such that $|\tilde{z}_{\nu}|\to1$ as $\nu\to\infty$ and 
\begin{equation}\frac{f^\star(\tilde{z}_{\nu})}{\varphi(|\tilde{z}_{\nu}|)}\geq c,\ 
\text{ for }\ \nu\in\mathbb N.\label{LPSPNHM.1}\tag{\ref*{T:LPSPNHM}.1} 
\end{equation}
If $f$ is not $\varphi$-normal, then the result follows from the Lohwater\,-\,Pommerenke Theorem for $\varphi$-normal harmonic mappings i.e., Result 
\ref{R:LPPNHM}. So we assume that $f$ is $\varphi$-normal and say that $|\tilde{z}_{\nu}|>1/2$, for $\nu\in\mathbb N$. We define
\begin{align*}
s_{\nu}=\frac{2|\tilde{z}_{\nu}|}{1+|\tilde{z}_{\nu}|}\ \ \text{and}\ \  s_{\nu}'=\frac{|\tilde{z}_{\nu}|}{2 -|\tilde{z}_{\nu}|}.
\end{align*}
Then $1/3<s_{\nu}'<|\tilde{z}_{\nu}|<s_{\nu}<1$ and $s_{\nu}\to 1$, $s_{\nu}'\to 1$ as $\nu\to \infty$.         
We have,
\begin{equation*}
s_{\nu}-s_{\nu}'=\frac{3|\tilde{z}_{\nu}|(1-|\tilde{z}_{\nu}|)}{(1+|\tilde{z}_{\nu}|)(2-|\tilde{z}_{\nu}|)}
\end{equation*}
and
\begin{equation*}
1-s_{\nu}=\frac{1-|\tilde{z}_{\nu}|}{1+|\tilde{z}_{\nu}|}.
\end{equation*}
This gives us,
\begin{equation}\label{LPSPNHM.2} \tag{\ref*{T:LPSPNHM}.2}
3(1-s_{\nu}) \geq \frac{3|\tilde{z}_{\nu}|(1-|\tilde{z}_{\nu}|)}{1+|\tilde{z}_{\nu}|} \geq \frac{3|\tilde{z}_{\nu}|(1-|\tilde{z}_{\nu}|)}{(1+|\tilde{z}_{\nu}|)(2-|\tilde{z}_{\nu}|)} = s_{\nu}-s_{\nu}'.
\end{equation} 
Take $\{z_{\nu}\}\subset \mathbb{D}$ such that
\begin{align*}
\begin{split}
M_{\nu} = &\max_{s_{\nu}'\leq|z|\leq s_{\nu}}\left(1-\frac{|z|}{s_{\nu}}\right)^{1/2}\left(1-\frac{s_{\nu}'}{|z|}\right)^{1/2}\frac{f^\star(z)}{\varphi(|z|)}\\
= & \left(1-\frac{|z_{\nu}|}{s_{\nu}}\right)^{1/2}\left(1-\frac{s_{\nu}'}{|z_{\nu}|}\right)^{1/2}\frac{f^\star(z_{\nu})}{\varphi(|z_{\nu}|)}.
\end{split} \label{LPSPNHM.3} \tag{\ref*{T:LPSPNHM}.3}
\end{align*}
Since $s_{\nu}' < |\tilde{z}_{\nu}| < s_{\nu}$, by (\ref{LPSPNHM.1}) we have,
\begin{align*}
\begin{split}
M_{\nu} &\geq \left(1-\frac{|\tilde{z}_{\nu}|}{s_{\nu}}\right)^{1/2} \left(1-\frac{s_{\nu}'}{|\tilde{z}_{\nu}|}\right)^{1/2} \frac{f^\star(\tilde{z}_{\nu})}{\varphi(|\tilde{z}_{\nu}|)} \\
&= \left(\frac{1-|\tilde{z}_{\nu}|}{2}\right)^{1/2} \left(\frac{1-|\tilde{z}_{\nu}|}{2-|\tilde{z}_{\nu}|}\right)^{1/2}  \frac{f^\star(\tilde{z}_{\nu})}{\varphi(|\tilde{z}_{\nu}|)} \\
&\geq \frac{1}{2}(1-|\tilde{z}_{\nu}|) \frac{f^\star(\tilde{z}_{\nu})}{\varphi(|\tilde{z}_{\nu}|)} \geq \frac{c}{4}>0.
\end{split}\label{LPSPNHM.4} \tag{\ref*{T:LPSPNHM}.4}
\end{align*}
Therefore $s_{\nu}'<|z_{\nu}|<s_{\nu}$, $|z_{\nu}|\to 1$ as $\nu\to\infty$ and by (\ref{LPSPNHM.3}) and (\ref{LPSPNHM.4}), $\displaystyle\frac{f^\star(z_{\nu})}{\varphi(|z_{\nu}|)}\to \infty$, as $\nu\to\infty$. Take
\begin{equation*}
\rho_{\nu} =\frac{1}{M_{\nu}}\left(1-\frac{|z_{\nu}|}{s_{\nu}}\right)^{1/2}\left(1-\frac{s_{\nu}'}{|z_{\nu}|}\right)^{1/2}=\frac{\varphi(|z_{\nu}|)}{f^\star(z_{\nu})}.
\end{equation*}
Then $\rho_{\nu}  \to 0$. From
\begin{equation*} \label{LPSPNHM.5} \tag{\ref*{T:LPSPNHM}.5}
1-|z_{\nu}|>1-s_{\nu}\geq\frac{1}{3}(s_{\nu}-s_{\nu}^\prime) \geq \frac{1}{3}(|z_{\nu}|-s_{\nu}^\prime).
\end{equation*}
Since $1-|z_{\nu}|>s_{\nu}-|z_{\nu}|$ and $1/3<s_{\nu}'<|z_{\nu}|<s_{\nu}$,
\begin{align*}
\begin{split}
\frac{\rho_{\nu} }{1 - |z_{\nu}|} &= \frac{1}{(1 - |z_{\nu}|)M_{\nu}} \left(1-\frac{|z_{\nu}|}{s_{\nu}}\right)^{1/2} \left(1-\frac{s_{\nu}'}{|z_{\nu}|}\right)^{1/2} \\
&\leq \frac{(s_{\nu} - |z_{\nu}|)^{1/2}(|z_{\nu}| - s_{\nu}')^{1/2}}{M_{\nu}(s_{\nu} - |z_{\nu}|)^{1/2}\left(\frac{1}{3}(|z_{\nu}| - s_{\nu}')\right)^{1/2} s_{\nu}^{1/2} |z_{\nu}|^{1/2}} \\
&= \frac{\sqrt{3}}{M_{\nu} s_{\nu}^{1/2} |z_{\nu}|^{1/2}} \leq \frac{12\sqrt{3}}{c} := \frac{1}{2R}.
\end{split}\label{LPSPNHM.6} \tag{\ref*{T:LPSPNHM}.6}
\end{align*}
The harmonic mappings $g_{\nu}(\zeta)=f\left(z_{\nu}+\frac{\rho_{\nu} \zeta}{\varphi(|z_{\nu}|)}\right)$ are defined for $|\zeta|<R$.
Then $g_{\nu}^\star(0)=\left(\dfrac{\rho_{\nu} }{\varphi(|z_{\nu}|)}\right)f^\star(z_{\nu})=1$. To 
show that the sequence $\{g_{\nu}(\zeta)\}$ is normal. Let us consider
\begin{align*} g_{\nu}^\#(\zeta)&=\frac{\rho_{\nu} }{\varphi(|z_{\nu}|)}f^\#\left(z_{\nu}+\frac{\rho_{\nu} \zeta}{\varphi(|z_{\nu}|)}\right)\\ &=\dfrac{\rho_{\nu} \varphi\left(\left|z_{\nu}+\frac{\rho_{\nu} \zeta}{\varphi(|z_{\nu}|)}\right|\right)}{\varphi(|z_{\nu}|)} \dfrac{f^\#\left(z_{\nu}+\frac{\rho_{\nu} \zeta}{\varphi(|z_{\nu}|)}\right)}{\varphi\left(\left|z_{\nu}+\frac{\rho_{\nu} \zeta}{\varphi(|z_{\nu}|)}\right|\right)}\\ &= \dfrac{\varphi\left(\left|z_{\nu}+\frac{\rho_{\nu} \zeta}{\varphi(|z_{\nu}|)}\right|\right)}{\varphi(|z_{\nu}|)} \left(\rho_{\nu} \dfrac{f^\#\left(z_{\nu}+\frac{\rho_{\nu} \zeta}{\varphi(|z_{\nu}|)}\right)}{\varphi\left(\left|z_{\nu}+\frac{\rho_{\nu} \zeta}{\varphi(|z_{\nu}|)}\right|\right)}\right)\to 1.\end{align*}
Thus $\{g_{\nu}(\zeta)\}$ is normal in $|\zeta|<R$ and has a convergent subsequence which converges to a 
harmonic mapping $g$ in $|\zeta|<R$ with $g^\#(0)=1\neq0$. So $g$ is non-constant. \smallskip

Conversely, suppose that $g_{\nu}(\zeta)=f\left(z_{\nu}+\frac{\rho_{\nu} \zeta}{\varphi(|z_{\nu}|)}\right)$ converges locally 
uniformly to $g(\zeta)$ in $|\zeta|<R$, where $g$ is non-constant and $\displaystyle\frac{\rho_{\nu} }{1-|z_{\nu}|}<\frac{1}{2R}$. Now, $g_{\nu}(\zeta)=f_{\nu}\left(z_{\nu}+\frac{\rho_{\nu} \zeta}{\varphi(|z_{\nu}|)}\right)$ are defined for 
$\displaystyle|\zeta|<R$.
\begin{align*}
    g_{\nu}^\star(\zeta)=&\frac{\rho_{\nu} }{\varphi(|z_{\nu}|)}f^\star\left(z_{\nu}+\frac{\rho_{\nu} \zeta}{\varphi(|z_{\nu}|)}\right)\\
    =& \dfrac{\rho_{\nu} \varphi\left(\left|z_{\nu}+\frac{\rho_{\nu} \zeta}{\varphi(|z_{\nu}|)}\right|\right)}{\varphi(|z_{\nu}|)} \dfrac{f^\star\left(z_{\nu}+\frac{\rho_{\nu} \zeta}{\varphi(|z_{\nu}|)}\right)}{\varphi\left(\left|z_{\nu}+\frac{\rho_{\nu} \zeta}{\varphi(|z_{\nu}|)}\right|\right)} \to 0\ \text{ as }\ \nu\to\infty.
\end{align*}
As $f$ is strongly $\varphi$-normal, so $|z|\to\partial\mathbb D$, \begin{equation*}\frac{f^\star(z)}{\varphi(|z|)}\to 0.\end{equation*} 
So $\{g_{\nu}\}$ is strongly $\varphi$-normal, thus $g_{\nu}^\star(\zeta)\to 0$, where $g$ is constant, which is a contradiction. 
\end{proof}
\begin{proof}[{\bf Proof of Theorem~\ref{T:LPSPNLHM}}]
The proof is identical to that of the previous theorem,
mutatis mutandis.
\end{proof}
\section*{Acknowledgement}
The authors would like to thank Prof. Sanjay Kumar Pant for their valuable discussions and suggestions. The second author is supported by Junior Research Fellowship from UGC (NTA Reference Number 241610000551).
\section*{Funding} The authors received no financial support for the research, authorship and/or publication of this article. 
\section*{Statement and Declaration}
\subsection*{Author Contribution} All authors contributed equally to this work.
\subsection*{Data Availability} Not applicable.
\subsection*{Conflict of Interest} The authors declare that they have no conflicts of interest related to the content of this article.
\subsection*{Consent to Participate} Not applicable.
\subsection*{Consent to Publish} Not applicable.
\bibliographystyle{amsplain}

\end{document}